\documentclass[a4paper]{amsart}
\usepackage{amsmath,amssymb,amscd,amsthm,enumerate}
\usepackage{subfigure,graphicx}

\newtheorem{Def}{Definition}[section]
\newtheorem{Th}[Def]{Theorem}
\newtheorem{Prop}[Def]{Proposition}

\newtheorem{Cor}[Def]{Corollary}
\newtheorem{Rem}[Def]{Remark}
\newtheorem{Ex}[Def]{Example}
\newtheorem{Nota}[Def]{Notation}

\newcommand{\R}{\mathbb{R}}
\newcommand{\Z}{\mathbb{Z}}

\newcommand{\C}{\mathbb{C}}

\newcommand{\CC}{\mathcal{C}}
\newcommand{\CE}{\mathcal{E}}
\newcommand{\DS}{\displaystyle }

\newcommand{\al}{\alpha }
\newcommand{\be}{\beta }
\newcommand{\ga}{\gamma }
\newcommand{\de}{\delta }
\newcommand{\Ga}{\Gamma }
\newcommand{\De}{\Delta }

\newcommand{\vep}{\varepsilon }
\newcommand{\vph}{\varphi }
\newcommand{\la}{\lambda }
\newcommand{\om}{\omega }
\newcommand{\na}{\nabla }
\newcommand{\pa}{\partial }

\newcommand{\re}{{\rm Re}}
\newcommand{\bu}{\bullet}
\newcommand{\si}{\sigma}

\title[Twisted cycles of $F_A$]
{Twisted period relations for Lauricella's hypergeometric function $F_A$}
\author[Y. Goto]{Yoshiaki GOTO}
\address{Department of Mathematics, Hokkaido University, Sapporo 060-0810, Japan }
\email{y-goto@math.sci.hokudai.ac.jp}

\dedicatory{}

\begin{document}

\begin{abstract}
We study Lauricella's hypergeometric function $F_A$ of $m$-variables 
and the system $E_A$ of differential equations annihilating $F_A$, 
by using twisted (co)homology groups. 
We construct twisted cycles with respect to an integral representation 
of Euler type of $F_A$. 
These cycles correspond to $2^m$ linearly independent solutions to 
$E_A$, which are expressed by hypergeometric series $F_A$. 
Using intersection forms of twisted (co)homology groups, we obtain 
twisted period relations which give quadratic relations for Lauricella's $F_A$.

\end{abstract}

\maketitle

\section{Introduction}\label{section-intro}
Lauricella's hypergeometric series $F_A$ of $m$-variables $x_1 ,\ldots ,x_m$ 
with complex parameters $a,b_1,\dots ,b_m ,c_1 ,\dots ,c_m$ is defined by 
\begin{align*}
  F_A (a,b,c ;x ) 
  =\sum_{n_{1} ,\ldots ,n_{m} =0} ^{\infty } 
  \frac{(a,n_1 +\cdots +n_m )(b_1,n_1) \cdots (b_m ,n_m)}
  {(c_1 ,n_1 )\cdots (c_m ,n_m ) n_1 ! \cdots n_m !} x_1 ^{n_1} \cdots x_m ^{n_m},
\end{align*}
where $x=(x_1 ,\ldots ,x_m),\ b=(b_1 ,\ldots ,b_m),\ c=(c_1 ,\ldots ,c_m)$, 
$c_1 ,\ldots ,c_m \not\in \{ 0,-1,-2,\ldots \}$ and $(c_1 ,n_1)=\Gamma (c_1+n_1)/\Gamma (c_1)$. 
This series converges in the domain 
$$
D_A :=\left\{ (x_1 ,\ldots ,x_m ) \in \C ^m  \ \middle| \ \sum _{k=1} ^{m} |x_k| <1  \right\} ,
$$
and admits the integral representation (\ref{integral}). 
The system $E_A (a,b,c)$ of differential equations annihilating $F_A (a,b,c;x)$ is a holonomic system 
of rank $2^m$ with the singular locus $S$ given in (\ref{sing-locus}). 
There is a fundamental system of solutions to $E_A (a,b,c)$ in a simply connected domain in $D_A -S$, 
which is given in terms of Lauricella's hypergeometric series $F_A$ with different parameters, 
see (\ref{series-sol}) for their expressions. 


In this paper, we construct $2^m$ twisted cycles which represent elements 
of the $m$-th twisted homology group 
concerning with the integral representation (\ref{integral}). 
They imply integral representations of the solutions (\ref{series-sol}) 
expressed by the series $F_A$. 
We evaluate the intersection numbers of these $2^m$ twisted cycles. 
Further, by using 
the intersection matrix of a basis of the twisted cohomology group in \cite{M-FA}, 
we give twisted period relations for two fundamental systems of $E_A$ 
with different parameters. 

In the study of twisted homology groups, twisted cycles 
given by bounded chambers are useful. 
For Lauricella's $F_A$, twisted cycles defined by $2^m$ bounded chambers 
are studied in \cite{MY-FA}. 
Though the integrals on these cycles are solutions to $E_A$, 
they do not give integral representations of 
the solutions (\ref{series-sol}), except for one cycle. 
We construct other twisted cycles from these $2^m$ bounded chambers 
by using a method introduced in \cite{G-FC}. 
For a subset $\{ i_1 ,\ldots ,i_r \}$ of $\{ 1,\ldots ,m \}$ of cardinality $r$, 
we construct a twisted cycle $\De_{i_1 \cdots i_r}$ from 
the direct product of an $r$-simplex and $(m-r)$ intervals, 
by a similar manner to \cite{G-FC}. 
See Section \ref{section-cycle}, for details. 
Our first main theorem states that this twisted cycle corresponds to 
the solution (\ref{series-sol}) 
expressed by the power function $\prod_{p=1}^r x_{i_p} ^{1-c_{i_p}}$ 
and the series $F_A$. 
Our construction has a simple combinatorial structure, and 
enables us to evaluate the intersection matrix formally. 
Once the intersection matrix for bases of twisted homology groups 
and that of twisted cohomology groups are 
evaluated, then we obtain twisted period relations 
which are originally identities among the integrals given by 
the pairings of elements of twisted homology and cohomology groups. 
Our first main theorem transforms these identities into 
quadratic relations among hypergeometric series $F_A$'s. 
Our second main theorem states these formulas in Section \ref{section-TPR}. 

As is in \cite{Beukers}, the irreducibility condition of the system $E_A (a,b,c)$ is 
known to be 
\begin{align*}
  b_1 ,\ldots ,b_m ,\ c_1 -b_1 ,\ldots ,c_m -b_m ,\ 
  a-\sum _{p=1}^{r} c_{i_r}  \not\in \Z 
\end{align*}
for any subset $\{ i_1 ,\ldots ,i_r \}$ of $\{ 1,\ldots ,m \}$. 
Since our interest is in the property of solutions to $E_A (a,b,c)$ 
expressed in terms of the hypergeometric series $F_A$, 
we assume throughout this paper that the parameters $a,\ b=(b_1, \ldots ,b_m)$  and 
$c=(c_1 ,\ldots ,c_m)$ 
satisfy the condition above and $c_1 ,\ldots ,c_m \not\in \Z$.

\section{Differential equations and integral representations}
In this section, we collect some facts about Lauricella's $F_A$ 
and the system $E_A$ of 
hypergeometric differential equations annihilating it. 

\begin{Nota}\label{k}
Throughout this paper, the letter $k$ always stands for an index running from $1$ to $m$. 
If no confusion is possible, 
$\DS \sum_{k=1} ^m$ and $\DS \prod_{k=1} ^m$ are often simply denoted by 
$\sum$ (or $\sum_k$) and $\prod$ (or $\prod_k$), respectively. 
For example, under this convention  
$F_A (a,b,c;x)$ is expressed as 
\begin{align*}
F_A (a,b,c ;x ) 
=\sum_{n_1 ,\ldots ,n_m =0} ^{\infty } 
 \frac{(a,\sum n_k ) \prod (b_k,n_k )}
 {\prod (c_k ,n_k )\cdot \prod n_k ! } \prod x_k ^{n_k} . 
\end{align*}
\end{Nota}

Let $\pa_k \ (k=1,\dots ,m)$ be the partial differential operator with respect to $x_k$. 
Lauricella's $ F_A (a,b,c;x)$ satisfies hypergeometric differential equations 
\begin{align*}
& \Bigl[
x_k (1-x_k )\pa _k ^{2} -x_k \sum _{\stackrel{1\leq i\leq m}{i\neq k}} x_i \pa _k \pa _i \\
& +(c_k -(a+b_k+1)x_k )\pa _k -b_k \sum _{\stackrel{1\leq i\leq m}{i\neq k}} x_{i} \pa _{i} -ab_k 
\Bigr] f(x)=0, 
\end{align*}
for $k=1,\ldots ,m$. 
The system generated by them is 
called Lauricella's system $E_A (a,b,c)$ of hypergeometric differential equations. 

\begin{Prop}[\cite{L}, \cite{Nakayama}] \label{solution}
The system $E_A (a,b,c)$ is a holonomic system of rank $2^m$ with the singular locus 
\begin{align}
\label{sing-locus}
S:= \left( 
\prod_{k=1}^m x_k \cdot 
\prod _{\{ i_1 ,\ldots ,i_r \} \subset \{ 1,\ldots ,m \}} \left( 1-\sum_{p=1}^r x_{i_p} \right) =0
\right) \subset \C^m . 
\end{align}
If $c_1 ,\ldots ,c_m \not\in \Z$, then 
the vector space of solutions to $E_A (a,b,c)$ in a simply connected domain in 
$D_A -S$ is spanned by the following $2^m$ elements: 
\begin{align}
\label{series-sol} 
f_{i_1 \cdots i_r} 
:=\left( \prod _{p=1} ^{r} x_{i_{p}} ^{1-c_{i_{p}}} \right) \cdot 
F_A \left( a+r-\sum _{p=1} ^{r} c_{i_p} ,b^{i_1 \cdots i_r} ,c^{i_1 \cdots i_r};x \right) .
\end{align}
Here $r$ runs from $0$ to $m$, indices $i_1 ,\ldots ,i_r$ satisfy $1\leq i_1 <\cdots <i_r \leq m$, 
and the row vectors $b^{i_1 \cdots i_r}$ and $c^{i_1 \cdots i_r}$ are defined by 
$$
b^{i_1 \cdots i_r} :=b+\sum_{p=1}^r (1-c_{i_p})e_{i_p} ,\quad
c^{i_1 \cdots i_r} :=c+2\sum_{p=1}^r (1-c_{i_p})e_{i_p} ,
$$
where $e_{i}$ is the $i$-th unit row vector of $\C^m$. 
\end{Prop}
For the above $i_1 ,\ldots ,i_r$, we take $j_{1} ,\ldots ,j_{m-r}$ so that 
$1\leq j_{1} <\cdots <j_{m-r} \leq m$ and $\{ i_{1} ,\ldots ,i_{r} ,j_{1} ,\ldots ,j_{m-r} \} =\{1,\ldots ,m \}$. 
It is easy to see that the $i_p$-th entries of $b^{i_1 \cdots i_r}$ and $c^{i_1 \cdots i_r}$ are 
$b_{i_p}-c_{i_p}+1$ and $2-c_{i_p}$ 
($1\leq p \leq r$) and 
the $j_q$-th entries are $b_{j_q}$ and $c_{j_q}$ ($1 \leq q \leq m-r$), respectively. 

We denote the multi-index `` $i_1 \cdots i_r$'' by a letter $I$ 
expressing the set $\{ i_1,\ldots ,i_r \}$. 
Note that the solution (\ref{series-sol}) for $r=0$ is 
$f(=f_{\emptyset})=F_A (a,b,c ;x)$. 

\begin{Prop}[Integral representation of Euler type, \cite{L}] \label{int-Euler}
For sufficiently small positive real numbers $x_1 ,\ldots ,x_m$, 
if $\re (c_k) >\re (b_k) >0 \ (k=1,\ldots ,m)$, 
then $F_A (a,b,c;x)$ admits the following integral representation: 
\begin{align}
  F_A (a,b,c ;x) =&\prod \frac{\Ga(c_k)}{\Ga(b_k)\Ga(c_k-b_k)} \label{integral} \\
  & \cdot \int _{(0,1)^m} \prod \left( t_k ^{b_k -1} \cdot (1-t_k)^{c_k -b_k -1} \right) 
  \cdot \left( 1-\sum x_k t_k \right) ^{-a} dt_1 \wedge \cdots \wedge dt_m. 
  \nonumber
\end{align}
\end{Prop}

\section{Twisted homology groups}\label{section-THG}
We review twisted homology groups and the intersection form 
between twisted homology groups in general situations, 
by referring to Chapter 2 of \cite{AK} and 
Chapters IV, VIII of \cite{KY}. 

For polynomials $P_j(t)=P_j(t_1 ,\ldots ,t_m ) \ (1 \leq j \leq n)$, 
we set $D_j :=\{ t \mid  P_j(t)=0 \} \subset \C^m$ and $M:=\C^m -(D_1 \cup \cdots \cup D_n)$.
We consider a multi-valued function $u(t)$ on $M$ defined as 
$$
u(t):=\prod _{j=1}^{n} P_j(t)^{\la _j} ,\ \la_j \in \C -\Z \ (1 \leq j \leq n).
$$
For a $k$-simplex $\si$ in $M$, we define a loaded $k$-simplex $\si \otimes u$ by
$\si$ loading a branch of $u$ on it. 
We denote the $\C$-vector space of finite sums of loaded $k$-simplexes by $\CC_k (M,u)$, 
called the $k$-th twisted chain group. 
An element of $\CC_k(M,u)$ is called a twisted $k$-chain. 
For a loaded $k$-simplex $\si \otimes u$ and a smooth $k$-form $\vph$ on $M$, 
the integral $\int_{\si \otimes u} u \cdot \vph$ is defined by 
$$
\int_{\si \otimes u} u \cdot \vph :=
\int_{\si} \left[ {\rm the \ fixed \ branch} \ {\rm of} \ u \ {\rm on} \ \si \right] \cdot \vph .
$$
By the linear extension of this, we define the integral on a twisted $k$-chain.  

We define the boundary operator $\pa^u :\CC_k (M,u) \to \CC_{k-1} (M,u)$ by
$$
\pa^u (\si \otimes u):= \pa(\si) \otimes u|_{\pa(\si)} , 
$$
where $\pa$ is the usual boundary operator and 
$u|_{\pa(\si)}$ is the restriction of $u$ to $\pa (\si)$. 
It is easy to see that $\pa^u \circ \pa^u =0$. 
Thus we have a complex 
$$
\CC_{\bu} (M,u) :\cdots 
\overset{\pa^u}{\longrightarrow} \CC_k (M,u)
\overset{\pa^u}{\longrightarrow} \CC_{k-1}(M,u)
\overset{\pa^u}{\longrightarrow} \cdots ,
$$
and its $k$-th homology group $H_k(\CC_{\bu} (M,u))$. 
It is called the $k$-th twisted homology group. 
An element of $\ker \pa^u$ is called a twisted cycle.

By considering $u^{-1} =1/u$ instead of $u$, we have $H_k(\CC_{\bu} (M,u^{-1}))$. 
There is the intersection pairing $I_h$ between $H_m(\CC_{\bu} (M,u))$ and $H_m(\CC_{\bu} (M,u^{-1}))$
(in fact, the intersection pairing is defined between $H_k(\CC_{\bu} (M,u))$ 
and $H_{2m-k}(\CC_{\bu} (M,u^{-1}))$, however we do not consider the cases $k \neq m$). 
Let $\De$ and $\De'$ be elements of $H_m(\CC_{\bu} (M,u))$ and $H_m(\CC_{\bu} (M,u^{-1}))$ given by 
twisted cycles $\sum_i \al_i \cdot \si_i \otimes u_i$ and $\sum_j \al'_j \cdot \si'_j \otimes u_j ^{-1}$ 
respectively,  
where $u_i$ (resp. $u_j ^{-1}$) is a branch of $u$ (resp. $u^{-1}$) on $\si_i$ (resp. $\si'_j$). 
Then their intersection number is defined by 
$$
I_h (\De ,\De'):=\sum_{i,j} \sum_{s \in \si_i \cap \si'_j} \al_i \al'_j 
\cdot (\si_i \cdot \si'_j)_s \cdot \frac{u_i (s)}{u_j(s)} ,
$$
where $(\si_i \cdot \si'_j)_s$ is the topological intersection number of 
$m$-simplexes $\si_i$ and $\si'_j$ at $s$. 

~

In this paper, we mainly consider 
$$
M:=\C ^{m} -\left( 
  \bigcup_k (t_k =0) \cup 
  \bigcup_k (1-t_k =0) \cup (v=0) \right) ,
$$
where $v:=1-\sum x_k t_k$. 
We consider the twisted homology group on $M$ with respect to the multi-valued function 
\begin{align*}
  u := \prod _{k=1} ^{m} t_{k} ^{b_k} (1-t_k)^{c_k -b_k -1}  \cdot v ^{-a} .
\end{align*}
Let $\De$ be the regularization of $(0,1)^m \otimes u$, 
which gives an element in $H_m (\CC_{\bu} (M,u))$. 
For the construction of regularizations, refer to 
Sections 3.2.4 and 3.2.5 of \cite{AK}. 
Proposition \ref{int-Euler} means that the integral 
$$
\int_{\De} u \vph ,\ \ 
\vph  :=\frac{dt_1 \wedge \cdots \wedge dt_m}{t_1 \cdots t_m }
$$
represents $F_A (a,b,c;x)$ modulo Gamma factors.

\section{Twisted cycles corresponding to local solutions 
$f_{i_1 \cdots i_r}$}\label{section-cycle}
In this section, we construct $2^m$ twisted cycles in $M$ corresponding to 
the solutions (\ref{series-sol}) to $E_A (a,b,c)$.

Let $0 \leq r \leq m$ and subsets 
$\{ i_1 ,\ldots ,i_r \}$ and $\{ j_1 ,\ldots ,j_{m-r} \}$ 
of $\{ 1,\ldots ,m \}$ satisfy 
$i_1 <\cdots <i_r ,\ j_{1} <\cdots <j_{m-r}$ and 
$\{ i_{1} ,\ldots ,i_{r} ,j_{1} ,\ldots ,j_{m-r} \} =\{1,\ldots ,m \}$. 
\begin{Nota}
From now on, the letter $p$ (resp. $q$) is always stands for an index running 
from $1$ to $r$ (resp. from $1$ to $m-r$). 
We use the abbreviations $\sum ,\ \prod$ for the indices 
$p,q$ as are mentioned in Notation \ref{k}. 
\end{Nota}

We set 
\begin{align*}
  & M _{i_1 \cdots i_r} \\
  & :=\C ^m -\left( \bigcup _k (s_k =0) 
    \cup \bigcup_p (s_{i_p} -x_{i_p} =0) \cup \bigcup_q (1-s_{j_q} =0)
    \cup (v_{i_1 \cdots i_r} =0)  \right) ,
\end{align*}
where 
\begin{align*}
  v_{i_1 \cdots i_r} :=1-\sum_p s_{i_p} -\sum_q x_{j_q} s_{j_q} .
\end{align*}
Let $u_{i_1 \cdots i_r}$ and $\vph_{i_1 \cdots i_r}$ be 
a multi-valued function and an $m$-form on $M_{i_1 \cdots i_r}$ defined as 
\begin{align*}
  & u_{i_1 \cdots i_r} :=
  \prod _{p=1} ^{r} s_{i_p} ^{b_{i_p}} \left( s_{i_p}-x_{i_p} \right) ^{c_{i_p}-b_{i_p}-1} 
  \cdot \prod _{q=1} ^{m-r} s_{j_q} ^{b_{j_q}} (1-s_{j_q})^{c_{j_q}-b_{j_q}-1}  
  \cdot v_{i_1 \cdots i_r} ^{-a}, \\
  & \vph _{i_1 \cdots i_r} :=\frac{ds_1 \wedge \cdots \wedge ds_m}{s_1 \cdots s_m} . 
\end{align*}
We construct a twisted cycle $\tilde{\De }_{i_1 \cdots i_r}$ in $M_{i_1 \cdots i_r}$ 
with respect to $u_{i_1 \cdots i_r}$. 
Note that if $\{ i_1 ,\ldots ,i_r \} =\emptyset$, then these settings coincide  
with those in the end of Section \ref{section-THG}. 
We choose positive real numbers $\vep _1 ,\ldots ,\vep _m$ and $\vep$ so that 
$\vep <1-\sum_{k} \vep _k$ and $\vep_k <\frac{1}{4}$. 
And let $x_1 ,\ldots ,x_m$ be small positive real numbers satisfying 
$$
x_k < \vep_k ,\quad \sum _k x_k (1+\vep_k) <\vep 
$$
(for example, if 
$$
\vep _k =\vep =\frac{1}{5m} ,\ 0<x_{k} <\frac{1}{6m^2} ,
$$
these conditions hold). 
Thus the closed subset 
\begin{align*}
  \sigma_{i_1 \cdots i_r} :=\left\{ (s_1 ,\ldots ,s_m )\in \R^m \ \Bigg| \ 
    \begin{array}{l}
      s_{i_p} \geq \vep_{i_p} ,\ 1-\sum s_{i_p} \geq \vep ,\\
      s_{j_q} \geq \vep_{j_q} ,\ 1-s_{j_q} \geq \vep_{j_q}
    \end{array}
  \right\} 
\end{align*}
is nonempty, since we have
$(\vep_1 \! +\! \frac{\de}{2m},\ldots ,\vep_m \! +\! \frac{\de}{2m}) \in \sigma_{i_1 \cdots i_r}$, 
where $\de :=1  - \sum \vep_k - \vep >0$. 
Further, $\sigma_{i_1 \cdots i_r}$ is contained in the bounded domain 
\begin{align*}
  \left\{ (s_1 ,\ldots ,s_m )\in \R^m \left|
      \begin{array}{l}
        s_{i_p} -x_{i_p} >0, \\
        0<s_{j_q} <1,
      \end{array} \right.
    1-\sum s_{i_p} -\sum x_{j_q} s_{j_q} >0 \right\} 
  \subset (0,1)^m.
\end{align*}
and is a direct product of an $r$-simplex and $(m-r)$ intervals. 
Indeed, $(s_1 ,\ldots ,s_m ) \in \sigma _{i_1 \cdots i_r} $ satisfies 
\begin{align*}
& s_{i_p} -x_{i_p} > s_{i_p} -\vep_{i_p} >0, \\
& 1-\sum s_{i_p} -\sum x_{j_q} s_{j_q} > \vep -\sum x_{j_q} >\vep -\sum x_k >  0 .
\end{align*}
The orientation of $\sigma _{i_1 \cdots i_r} $ 
is induced from the natural embedding $\R^m \subset \C^m$. 
We construct a twisted cycle from $\sigma _{i_1 \cdots i_r}$. 
We may assume that $\vep _k =\vep$ (the above example satisfies this condition), and denote them by $\vep$. 
Set $L_1 :=(s_1 =0),\ldots ,\ L_m :=(s_m =0) ,\ 
L_{m+1} :=(1-s_{j_1}=0),\ldots ,\ L_{2m-r} :=(1-s_{j_{m-r}}),\ L_{2m-r+1} :=(1-\sum s_{i_p} =0)$, 
and let $U(\subset \R ^m )$ be a bounded chamber surrounded by $L_1 ,\ldots ,\ L_{2m-r+1}$, 
then $\sigma _{i_1 \cdots i_r}$ is contained in $U$. 
Note that we do not consider the hyperplane $L_{2m-r+1}$ 
(resp. the hyperplanes $L_{m+1} ,\ldots ,\ L_{2m-r}$), when $r=0$ (resp. $r=m$). 
For $J\subset \{ 1\ldots ,2m-r+1 \}$, we consider $L_J :=\cap _{j\in J} L_j ,\ U_J :=\overline{U} \cap L_J$ 
and $T_J :=\vep $-neighborhood of $U_J$. 
Then we have 
$$
\sigma _{i_1 \cdots i_r} =U-\bigcup _J T_J .
$$ 
Using these neighborhoods $T_J$, we can construct a twisted cycle $\tilde{\De} _{i_{1} \cdots i_{r}}$ 
in the same manner as Section 3.2.4 of \cite{AK} 
(notations $L$ and $U$ correspond to $H$ and $\De$ in \cite{AK}, respectively). 
Note that we have to consider contribution of branches of 
$\DS s_{i_p} ^{b_{i_p}} \left( s_{i_p}-x_{i_p} \right) ^{c_{i_p}-b_{i_p}-1}$, 
when we deal with the circle associated to $L_{i_p} \ (p=1,\dots ,r)$,  
because of $x_{i_p} <\vep$. 
Thus the exponent about this contribution is 
$$
b_{i_p} +(c_{i_p}-b_{i_p}-1)=c_{i_p} -1.
$$
The exponents about the contributions 
of the circles associated to $L_{j_q} ,\ L_{m+q} ,\ L_{2m-r+1}$ are simply 
$$
b_{j_q}  ,\ c_{j_q} -b_{j_q}-1,\ -a, 
$$ 
respectively. 
We briefly explain the expression of $\tilde{\De}_{i_1 \cdots i_r}$. 
For $j=1,\ldots ,2m-r+1$, let $l_j$ be the $(m-1)$-face of $\si_{i_1 \cdots i_r}$ 
given by $\si_{i_1 \cdots i_r} \cap \overline{T_j}$, 
and let $S_j$ be a positively oriented circle with radius $\vep$ 
in the orthogonal complement of $L_j$ starting from the projection of 
$l_j$ to this space and surrounding $L_j$. 
Then $\tilde{\De} _{i_{1} \cdots i_{r}}$ is written as
\begin{align*}
  \sigma _{i_{1} \cdots i_{r}} 
  +\sum _{ \emptyset \neq J \subset \{ 1,\ldots , 2m-r+1\}} 
  \left( \prod_{j\in J} \frac{1}{d_j} \right) 
  \cdot \left( \Biggl( \bigcap_{j\in J} l_j \Biggr) \times \prod_{j\in J} S_j \right) ,
\end{align*}
where 
\begin{align*}
  d_{i_p} :=\ga_{i_p} -1,\ d_{j_q} :=\be_{j_q} -1 ,\ 
  d_{m+q} :=\ga_{j_q} \be_{j_q}^{-1} -1,\ d_{2m-r+1} :=\al^{-1} -1,
\end{align*}
and $\al :=e^{2\pi \sqrt{-1}a} ,\ \be_k :=e^{2\pi \sqrt{-1}b_k},\ 
\ga_k :=e^{2\pi \sqrt{-1}c_k}$. 
Note that we define an appropriate orientation for each 
$(\cap_{j\in J} l_j)\times \prod_{j\in J} S_j$, 
see Section 3.2.4 of \cite{AK} for details.

\begin{Ex}
  We give explicit forms of $\tilde{\De},\ \tilde{\De}_1$ and $\tilde{\De}_{12}$, for $m=2$. 
  \begin{enumerate}[(i)]
  \item In the case of $I=\emptyset$, $\tilde{\De}$ is the usual regularization 
    of $(0,1)^m$. 
  \item In the case of $I=\{ 1 \}$, we have 
    \begin{align*}
      \tilde{\De}_1 =&\si_1 +\frac{S_1 \times l_1 }{1-\ga_1} 
      +\frac{S_2 \times l_2}{1-\be_2}  
      +\frac{S_4 \times l_4}{1-\al^{-1}} 
      +\frac{S_3 \times l_3}{1-\ga_2 \be_2^{-1}} \\
      &+\frac{S_1 \times S_2}{(1-\ga_1)(1-\be_2)} 
      +\frac{S_2 \times S_4}{(1-\be_2)(1-\al^{-1})} \\   
      &+\frac{S_4 \times S_3}{(1-\al^{-1})(1-\ga_2 \be_2^{-1})}
      +\frac{S_3 \times S_1}{(1-\ga_2 \be_2^{-1})(1-\ga_1)} , 
    \end{align*}
    where the $1$-chains $l_j$ satisfy $\pa \si =\sum_{j=1}^4 l_j$ 
    (see Figure \ref{fig1}), 
    and the orientation of each direct product is induced 
    from those of its components. 
    \begin{figure}[h]
      \centering{
        \includegraphics[scale=0.8]{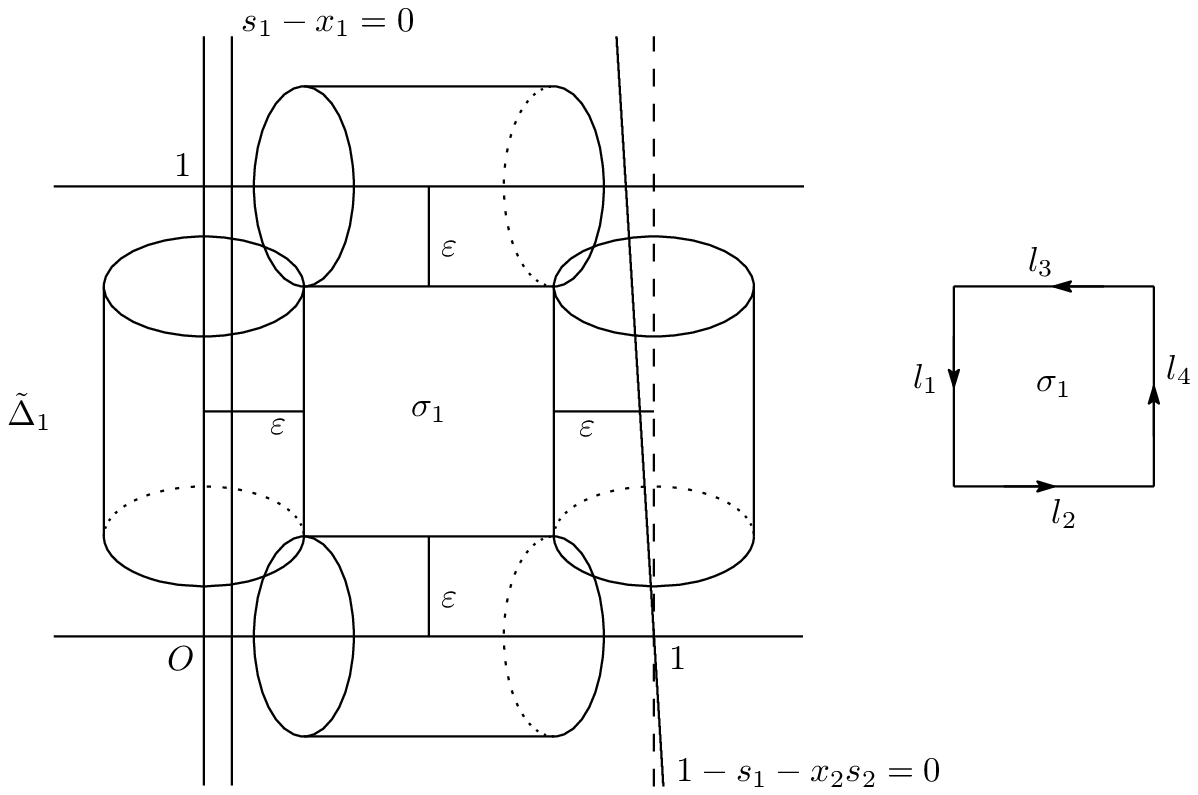} 
      }
      \caption{$\tilde{\De}_1$ for $m=2$. \label{fig1}}
    \end{figure}
  \item In the case of $I=\{ 1,2 \}$, we have 
    \begin{align*}
      \tilde{\De}_{12} =&\si_{12} +\frac{S_1 \times l_1 }{1-\ga_1} 
      +\frac{S_2 \times l_2}{1-\ga_2}  
      +\frac{S_3 \times l_3}{1-\al^{-1}}  \\
      &+\frac{S_1 \times S_2}{(1-\ga_1)(1-\ga_2)} 
      +\frac{S_2 \times S_3}{(1-\ga_2)(1-\al^{-1})}   
      +\frac{S_3 \times S_1}{(1-\al^{-1})(1-\ga_1)} , 
    \end{align*}
    where the $1$-chains $l_j$ satisfy $\pa \si =l_1 +l_2 +l_3$ 
    (see Figure \ref{fig2}), 
    and the orientation of each direct product is induced 
    from those of its components. 
    \begin{figure}[h]
      \centering{
        \includegraphics[scale=0.8]{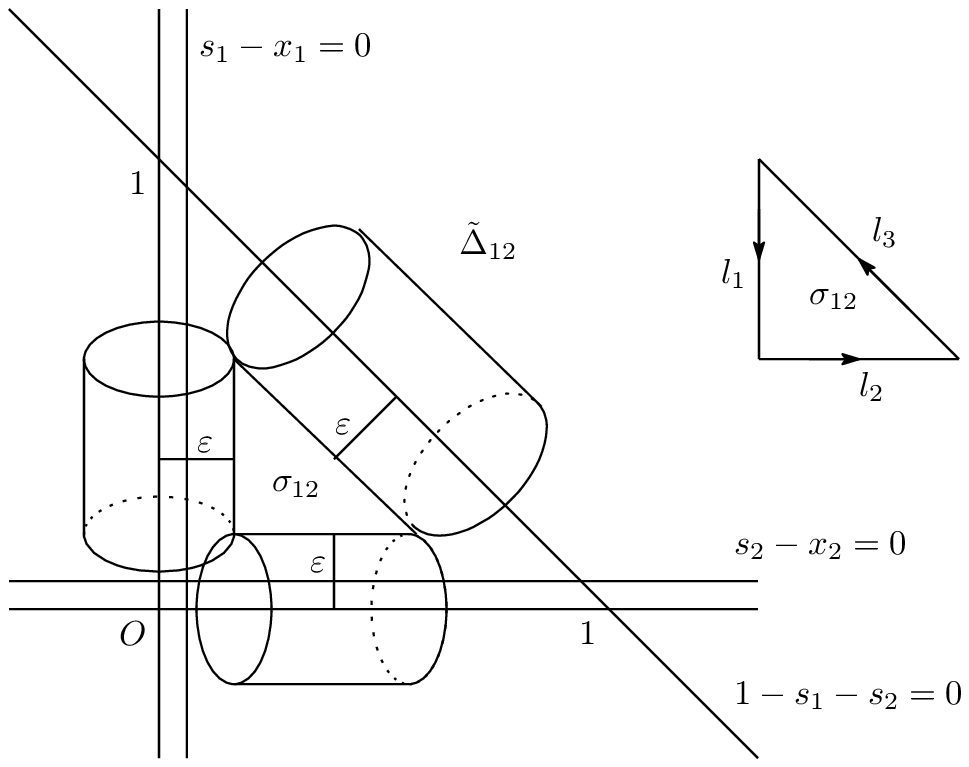} 
      }
      \caption{$\tilde{\De}_{12}$ for $m=2$. \label{fig2}}
    \end{figure}
  \end{enumerate}
\end{Ex}

We consider the following integrals: 
\begin{align*}
& F _{i_1 \cdots i_r} := 
\int _{\tilde{\De }_{i_1 \cdots i_r}} u_{i_1 \cdots i_r} \vph_{i_1 \cdots i_r} \\ 
&= \int _{\tilde{\De }_{i_1 \cdots i_r}}
\prod _{p=1} ^{r} s_{i_p} ^{c_{i_p} -2} \left( 1-\frac{x_{i_p}}{s_{i_p}} \right) ^{c_{i_p}-b_{i_p}-1} 
\cdot \prod _{q=1} ^{m-r} s_{j_q} ^{b_{j_q}-1} (1-s_{j_q})^{c_{j_q}-b_{j_q}-1} \\
& \hspace{20mm} \cdot \left( 1-\sum _{p=1} ^{r} s_{i_p} -\sum _{q=1} ^{m-r} x_{j_q} s_{j_q}  \right) ^{-a} 
ds_1 \wedge \cdots \wedge ds_m .
\end{align*}

\begin{Prop}\label{series-cycle}
\begin{align*}
F _{i_1 \cdots i_r} 
=&\prod _{p=1} ^{r} \Ga (c_{i_p} -1) 
\cdot \prod _{q=1} ^{m-r} \frac{\Ga(b_{j_q}) \Ga(c_{j_q} -b_{j_q})}{\Ga(c_{j_q})} 
\cdot \frac{\Ga (1-a)}{\Ga (\sum c_{i_p} -a-r+1)} \\
& \cdot F_A (a+r-\sum _{p=1} ^{r} c_{i_p} ,b^{i_1 \cdots i_r},c^{i_1 \cdots i_r} ;x).
\end{align*}
\end{Prop}
\begin{proof}
We compare the power series expansions of the both sides. 
Note that the coefficient of $x_1 ^{n_1} \cdots x_m ^{n_m}$ in the series expression of 
$F_A (a+r-\sum _{p=1} ^{r} c_{i_p} ,b^{i_1 \cdots i_r},c^{i_1 \cdots i_r} ;x)$ is 
\begin{align*}
A_{n_1 \dots n_m} :=& 
\frac{\Ga (a+r-\sum_p c_{i_p} +\sum_k n_k )}{\Ga (a+r-\sum_p c_{i_p} )}
\cdot \prod_p \frac{\Ga (b_{i_p}+1-c_{i_p} +n_{i_p} )}{\Ga (b_{i_p}+1-c_{i_p})} 
\cdot \prod_q \frac{\Ga (b_{j_q}+n_{j_q} )}{\Ga (b_{j_q})} \\
& \cdot \prod_p \frac{\Ga (2-c_{i_p})}{\Ga (2-c_{i_p} +n_{i_p})} 
\cdot \prod_q \frac{\Ga (c_{j_q})}{\Ga (c_{j_q} +n_{j_q})} 
\cdot \prod_k \frac{1}{n_k !}. 
\end{align*}
On the other hand, we have
\begin{align*}
  \left( 1-\frac{x_{i_p}}{s_{i_p}} \right) ^{c_{i_p} -b_{i_p} -1} 
  =\sum _{n_{i_p}} \frac{\Ga (b_{i_p }-c_{i_p} +1+ n_{i_p} )}
  {\Ga (b_{i_p }-c_{i_p} +1) \cdot n_{i_p} !} 
  s_{i_p} ^{-n_{i_p}}  x_{i_p} ^{n_{i_p}}
\end{align*}
and
\begin{align*}
  & \left(  1-\sum _{p=1} ^{r}  s_{i_{p}} -\sum _{q=1} ^{m-r} x_{j_{q}} s_{j_{q}}  \right) ^{-a} \\
  &= \sum _{n_{j_1} ,\dots ,n_{j_{m-r}}} \frac{\Ga (a+\sum n_{j_q} )}{\Ga (a) \cdot \prod n_{j_q} !} 
  (1-\sum s_{i_p})^{-a-\sum n_{j_q}} \cdot \prod s_{j_q}^{n_{j_q}}  x_{j_q} ^{n_{j_q}}. 
\end{align*}
When $r=0$ (resp. $r=m$), we do not need the first (resp. second) expansion. 
The convergences of these power series expansions are verified as follows. 
By the construction of $\tilde{\De }_{i_{1} \cdots i_{r}} $, we have 
$$
0<x_k <\vep_k ,\ \vep_{i_p} \leq |s_{i_p}| ,\  
|s_{j_q}| \leq 1+\vep_{j_q} ,\ \big|1-\sum s_{i_p} \big| \geq \vep .
$$
Thus the uniform convergences on $\tilde{\De }_{i_{1} \cdots i_{r}} $ follow from 
\begin{align*}
  & \left| \frac{x_{i_p}}{s_{i_{p}}}  \right| <\frac{\vep_{i_p}}{\vep_{i_p}} =1 , \\
  & \left| \frac{1}{1-\sum s_{i_p}} \cdot \sum x_{j_q} s_{j_q} \right| 
  \leq \frac{1}{|1-\sum s_{i_p}|} \cdot  \sum x_{j_q} |s_{j_q}| 
  \leq \frac{1}{\vep} \cdot \sum x_{j_q} (1+ \vep_{j_q}) < \frac{\vep}{\vep} =1.
\end{align*}
Since $\tilde{\De}_{i_1 \cdots i_r}$ is constructed as a finite sum of 
loaded (compact) simplexes, 
we can exchange the sum and the integral in the expression of $F_{i_1 \cdots i_r}$. 
Then the coefficient of $x_1 ^{n_1} \cdots x_m ^{n_m}$ in the series expansion of $F _{i_{1} \cdots i_{r}}$ is
\begin{align}
  \label{Bn}
  & B_{n_1 \dots n_m}:= 
  \prod_p \frac{\Ga (b_{i_p }-c_{i_p} +1+ n_{i_p} )}{\Ga (b_{i_p }-c_{i_p} +1)} 
  \cdot \frac{\Ga (a+\sum n_{j_q} )}{\Ga (a) } \cdot \prod_k \frac{1}{n_k !} \\
  & \cdot \int _{\tilde{\De }_{i_1 \cdots i_r}} \prod_p s_{i_p} ^{c_{i_p} -2 -n_{i_p}} 
  \cdot (1-\sum s_{i_p})^{-a-\sum n_{j_q}} 
  \cdot \prod_q s_{j_q} ^{b_{j_q}-1 -n_{j_q}} (1-s_{j_q}) ^{c_{j_q} -b_{j_q} -1} ds .
  \nonumber
\end{align}
By the construction, the twisted cycle $\tilde{\De }_{i_1 \cdots i_r}$ of this integral is 
identified with the usual regularization of the domain 
\begin{align*}
  \left\{ (s_1 ,\ldots ,s_m )\in \R^m \ \mid  
    s_{i_p} >0 ,\ 1-\sum s_{i_p} >0 ,\ 
    0<s_{j_q} <1 
  \right\} 
\end{align*}
for the multi-valued function 
$$
\prod_p s_{i_p} ^{c_{i_p} -1 -n_{i_p}} (1-\sum s_{i_p})^{-a-\sum n_{j_q}} 
\cdot \prod_q s_{j_q} ^{b_{j_q} -n_{j_q}} (1-s_{j_q}) ^{c_{j_q} -b_{j_q} -1}
$$
on $\C^m -\left( \bigcup_k (s_k=0) \cup \bigcup_q (1-s_{j_q} =0) \cup (1- \sum s_{i_p} =0) \right)$. 
Hence the integral in (\ref{Bn}) is equal to
\begin{align*}
\frac{\prod_p \Ga (c_{i_p}-n_{i_p}-1) \cdot \Ga (-a-\sum n_{j_q}+1)}{\Ga (\sum c_{i_p} -a-\sum n_k -r+1)} 
\cdot \prod_q \frac{\Ga (b_{j_q}+n_{j_q}) \Ga (c_{j_q} -b_{j_q})}{\Ga (c_{j_q} +n_{j_q})} .
\end{align*}
Using the formula 
\begin{align}
  \Ga (z) \Ga (1-z) =\frac{\pi}{\sin (\pi z)} , \label{Gamma}
\end{align}
we thus have 
\begin{align*}
  \frac{ B_{n_1 \dots n_m}}{A_{n_1 \dots n_m}} =
  \prod_p \Ga (c_{i_p} -1) 
  \cdot \prod_q \frac{\Ga (b_{j_q}) \Ga(c_{j_q} -b_{j_q})}{\Ga (c_{j_q}) } 
  \cdot \frac{\Ga(1-a)}{\Ga(\sum c_{i_p}-a-r+1)} ,
\end{align*}
which implies the proposition. 
\end{proof}

We define a bijection 
$\iota _{i_1 \cdots i_r} :M_{i_1 \cdots i_r} \rightarrow M$ by
$$
\iota _{i_1 \cdots i_r} (s_1 ,\ldots ,s_m ):=(t_1 ,\ldots ,t_m );\ 
t_{i_p} =\frac{s_{i_p}}{x_{i_p}} ,\ t_{j_q} =s_{j_q} .
$$
For example, $\iota (=\iota _{\emptyset })$ is the identity map on $M=M_{\emptyset}$. 

We also define branches of the multi-valued function $u$ 
on real bounded chambers in $M$. On the domain 
\begin{eqnarray*}
D_{i_1 \cdots i_r} :=\{ (t_1 ,\ldots ,t_r) \in \R^m \mid 
t_k >0,\ 1-\sum x_k t_k >0,\ 1-t_{i_p} <0 ,\ 1-t_{j_q} >0 \} ,
\end{eqnarray*}
the arguments of $t_k ,\ 1-\sum x_k t_k ,\ 1-t_{i_p}$ and $1-t_{j_q}$ are given as follows. 
\begin{eqnarray*}
  \begin{array}{|c|c|c|c|}
    \hline
    t_k & 1-\sum x_k t_k & 1-t_{i_p} & 1-t_{j_q} \\ \hline
    0 & 0 & -\pi & 0 \\ \hline
  \end{array}
\end{eqnarray*}
We state our first main theorem.

\begin{Th}\label{solutions-cycle}
We define a twisted cycle $\De _{i_1 \cdots i_r}$ in $M$ by 
\begin{align*}
  \De _{i_{1} \cdots i_{r}} :=(\iota _{i_1 \cdots i_r} )_{*} (\tilde{\De }_{i_1 \cdots i_r}) .
\end{align*}
Then we have
\begin{align*}
  & \int _{\De _{i_1 \cdots i_r} } \prod 
  \left( t_{k} ^{b_k-1} \cdot (1-t_k)^{c_k -b_k -1} \right) 
  \cdot \left( 1-\sum x_k t_k \right) ^{-a} dt_1 \wedge \cdots \wedge dt_m \\
  & \left( =\int_{\De_{i_1 \cdots i_r}} u \vph  \right)
  =e^{\pi \sqrt{-1} (\sum b_{i_p} -\sum c_{i_p} +r)} 
  \prod _{p=1} ^{r} x_{i_{p}} ^{1-c_{i_{p}}} \cdot F _{i_1 \cdots i_r} ,
\end{align*}
and hence this integral corresponds to the local solution $f_{i_1 \cdots i_r} $ to 
$E_A (a,b,c)$ given in Proposition \ref{solution}.
\end{Th}
\begin{proof}
Since $\iota _{i_1 \cdots i_r} (\sigma _{i_1 \cdots i_r}) \subset D_{i_1 \cdots i_r}$, 
the left hand side is equal to 
\begin{align*}
  e^{\pi \sqrt{-1} (\sum b_{i_p} -\sum c_{i_p} +r)} 
  \cdot & \int _{\De _{i_{1} \cdots i_{r}} }  
  \prod_p \left( t_{i_p} ^{b_{i_p}-1} \cdot (t_{i_p} -1)^{c_{i_p} -b_{i_p} -1} \right) \\
  & \cdot \prod_q \left( t_{j_q} ^{b_{j_q}-1} \cdot (1-t_{j_q})^{c_{j_q} -b_{j_q} -1} \right) 
  \cdot \left( 1-\sum x_k t_k \right) ^{-a} dt_1 \wedge \cdots \wedge dt_m  ,
\end{align*}
where the branch of the integrand is determined naturally.  
Pulling back this integral by $\iota _{i_1 \cdots i_r}$ leads the first claim. 
This and Proposition \ref{series-cycle} imply the second claim. 
\end{proof}
\begin{Rem}
  Except in the case of $\{ i_1 ,\ldots ,i_r \} =\emptyset$, 
  the twisted cycle $\De_{i_1 \cdots i_r}$ is different from 
  the regularization of $D_{i_1 \cdots i_r}\otimes u$ 
  as elements in $H_m (\CC_{\bu} (M,u))$. 
\end{Rem}

The replacement $u \mapsto u^{-1} =1/u$ and the construction same as $\De _{i_1 \cdots i_r}$ give 
the twisted cycle $\De_{i_1 \cdots i_r} ^{\vee}$ 
which represents an element in $H_m (\CC_{\bu} (M,u^{-1}))$. 
We obtain the intersection numbers of the twisted cycles $\{ \De _{i_1 \cdots i_r} \}$ and 
$\{ \De _{i_1 \cdots i_r}^{\vee} \}$. 

\begin{Th}\label{H-intersection}
  \begin{enumerate}[(i)]
  \item For $I,J \subset \{ 1,\ldots ,m \}$ such that $I \neq J$, 
    we have $I_h (\De_I ,\De_J ^{\vee}) =0$. 
  \item The self-intersection number of $\De _{i_1 \cdots i_r}$ is 
  \begin{align*}
    I_{h} (\De _{i_{1} \cdots i_{r}} ,\De _{i_{1} \cdots i_{r}}^{\vee} )
    = \frac{\al -\prod_p \ga_{i_p}}{(\al -1) \prod_p (1-\ga _{i_p})} 
    \cdot \prod_q \frac{\be_{j_q} (1-\ga_{j_q})}{(1-\be_{j_q})(\be_{j_q} -\ga_{j_q})} .
  \end{align*}
  \end{enumerate}
\end{Th}
\begin{proof}
(i) Since $\De _{i_1 \cdots i_r} $'s represent local solutions (\ref{series-sol}) 
to $E_A (a,b,c)$ by Theorem \ref{solutions-cycle}, 
this claim is followed from similar arguments to the proof of Lemma 4.1 in \cite{GM}. \\
(ii) By $\iota _{i_1 \cdots i_r}$, the self-intersection number of
$\De _{i_1 \cdots i_r}$ is equal to that of $\tilde{\De }_{i_1 \cdots i_r}$ 
with respect to the multi-valued function $u_{i_1 \cdots i_r}$. 
To calculate this, we apply results in \cite{KY}. 
Since we construct the twisted cycle $\tilde{\De }_{i_1 \cdots i_r}$ from 
the direct product of an $r$-simplex and $(m-r)$ intervals, 
the self-intersection number of $\tilde{\De }_{i_1 \cdots i_r}$ 
is obtained as the product of those of the simplex and the intervals. 
Thus we have
\begin{align*}
  I_{h} (\De _{i_1 \cdots i_r} ,\De _{i_1 \cdots i_r}^{\vee} )
  = \frac{1-\prod_{p}\ga_{i_p} \cdot \al^{-1}}{\prod_{p}(1-\ga_{i_p})\cdot (1-\al^{-1})}
  \cdot \prod_q \frac{1-\ga_{j_q}}{(1-\be_{j_q})(1-\ga_{j_q} \be_{j_q}^{-1})} .
\end{align*}
\end{proof}

\section{Intersection numbers of twisted cohomology groups}
\label{section-cohomology}
In this section, we review twisted cohomology groups and 
the intersection form between twisted cohomology groups in our situation, 
and collect some results of \cite{M-FA} in which intersection numbers 
of twisted cocycles are evaluated. 

Recall that 
\begin{align*}
  & M=\C^m -\left(   \bigcup_k (t_k =0) \cup 
  \bigcup_k (1-t_k =0) \cup (v=0) \right), \\
  & u= \prod t_k ^{b_k} (1-t_k)^{c_k-b_k-1} \cdot v^{-a}.
\end{align*}
We consider the logarithmic $1$-form 
$$
\om :=d\log u =\frac{du}{u} .
$$
We denote the $\C$-vector space of smooth $k$-forms on $M$ by $\CE^k (M)$. 
We define the covariant differential operator
$\na_{\om} :\CE^k (M) \to \CE^{k+1} (M)$ by
$$
\na_{\om} (\psi):= d\psi +\om \wedge \psi ,\ \ \psi \in \CE^{k} (M).
$$
Because of $\na_{\om} \circ \na_{\om} =0$, we have a complex 
$$
\CE^{\bu} (M) :\cdots 
\overset{\na_{\om}}{\longrightarrow} \CE^k (M)
\overset{\na_{\om}}{\longrightarrow} \CE^{k+1}(M)
\overset{\na_{\om}}{\longrightarrow} \cdots ,
$$
and its $k$-th cohomology group $H^k(M,\na_{\om})$. 
It is called the $k$-th twisted de Rham cohomology group. 
An element of $\ker \na_{\om}$ is called a twisted cocycle.
By replacing $\CE^k (M)$ with the $\C$-vector space $\CE_c ^k (M)$ of 
smooth $k$-forms on $M$ with compact support, we obtain 
the twisted de Rham cohomology group 
$H_c ^k (M,\na_{\om})$ with compact support. 
By \cite{Cho}, we have $H^k(M,\na_{\om})=0$ for all $k \neq m$. 
Further, by Lemma 2.9 in \cite{AK}, 
there is a canonical isomorphism 
$$
\jmath : H^m (M,\na_{\om}) \to H_c ^m (M,\na_{\om}) .
$$ 
By considering $u^{-1} =1/u$ instead of $u$, we have 
the covariant differential operator $\na_{-\om}$ and 
the twisted de Rham cohomology group $H^k (M,\na_{-\om})$. 
The intersection form $I_c$ between $H^m (M,\na_{\om})$ and $H^m (M,\na_{-\om})$ 
is defined by 
$$
I_c (\psi ,\psi'):=\int_M \jmath (\psi) \wedge \psi' ,\quad 
\psi \in H^m(M,\na_{\om}),\ \psi' \in H^m (M,\na_{-\om}),
$$
which converges because of the compactness of the support of $\jmath (\psi)$.

\begin{Rem}
  By Lemma 2.8 and Theorem 2.2 in \cite{AK}, we have 
  \begin{align*}
    & \dim H_k (\CC_{\bu} (M,u))=0 \ \  (k \neq m),\\
    & \dim H_m (\CC_{\bu} (M,u))=\dim H^m(M,\na_{\om}) =(-1)^m  \chi (M)=2^m , 
  \end{align*}
  where $\chi (M)$ is the Euler characteristic of $M$. 
  Under our assumption for the parameters $a,\ b$ and $c$ (see Section \ref{section-intro}), 
  since the determinant of the intersection matrix $(I_h (\De_I ,\De_J ^{\vee}))$ 
  is not zero by Theorem \ref{H-intersection}, 
  the twisted cycles $\{ \De_I \} _I$ form a basis of $H_m (\CC_{\bu} (M,u))$.   
\end{Rem}

The intersection numbers of some twisted cocycles are evaluated in \cite{M-FA}. 
We use a part of these results. 
We consider $m$-forms 
$$
\vph^{i_1 \cdots i_r} :=\frac{dt_1 \wedge \cdots \wedge dt_m}
{\prod_p (t_{i_p} -1) \cdot \prod_q t_{j_q}}  
$$
on $M$, which is denoted by $\vph_{x,(v_1,\ldots ,v_m)}$ 
with $v_{i_p} =1,\ v_{j_q} =0$ in \cite{M-FA}. 
Note that $\vph =\vph^{\emptyset}$ is equal to 
$\vph =\vph_{\emptyset}$ defined in Section \ref{section-THG} 
(and \ref{section-cycle}). We put
\begin{eqnarray*}
A_{i_1 \cdots i_r} =A_{I}
  :=\sum _{\{ I^{(l)} \}} \prod_{l=1}^{r} \frac{1}{a-\sum c_{i_p^{(l)}} +l} , 
\end{eqnarray*}
where $\{ I^{(l)} \}$ runs sequences of subsets of $I=\{ i_1 , \ldots ,\ i_r \}$, which satisfy 
$$
I=I^{(r)} \supsetneq I^{(r-1)} \supsetneq \cdots \supsetneq I^{(2)} \supsetneq I^{(1)} \neq \emptyset ,
$$
and we write $I^{(l)} =\{ i_{1} ^{(l)} ,\ldots ,i_{l} ^{(l)} \}$. 

\begin{Prop}[\cite{M-FA}] \label{C-intersection}
  We have 
  \begin{align*}
    I_{c} (\vph^I ,\vph^{I'}) 
    = (2\pi \sqrt{-1}) ^{m} \cdot 
    \sum _{N \subset \{ 1 ,\ldots ,\ m \}} \left( A_N 
      \prod_{n \not\in N} \frac{\de _{I,I'}(n)}{\tilde{b}_{I}(n)}   \right) ,
  \end{align*}
  where  
  \begin{align*}
    \de _{I,I'}(n):=&\left\{
      \begin{array}{l}
        1 \ \ (n \in (I \cap I')\cup (I^c \cap {I'}^c) ) \\
        0 \ \ ({\rm otherwise}) ,
      \end{array}
    \right. \\
    \tilde{b}_{I}(n):=&\left\{
      \begin{array}{l}
        c_n -b_n -1 \ \ (n \in I ) \\
        b_n \ \ (n \in I^c ) .
      \end{array}
    \right. 
  \end{align*}
  Under our assumptions for the parameters, 
  $\{ \vph^I \} _I$ form a basis of $H^m (M,\na_{\om})$. 
\end{Prop}

\section{Twisted period relations} \label{section-TPR}
Because of the compatibility of intersection forms and pairings obtained by integrations 
(see \cite{CM}), we have the following theorem. 
\begin{Th}[Twisted period relations, \cite{CM}] \label{TPR}
  We have 
  \begin{align}
    \label{TPR-eq}
    I_c (\vph ^{I} ,\vph ^{I'}) 
    = \sum _{N \subset \{ 1,\ldots ,m \}} \frac{1}{I_h (\De _N ,\De _{N}^{\vee} )}
    \cdot g_{I,N} \cdot g_{I',N} ^{\vee} ,
  \end{align}
  where 
  \begin{eqnarray*}
    g_{I,N} =\int _{\De _{N}} u\vph^I ,\ 
    g_{I',N} ^{\vee} = \int _{\De _{N} ^{\vee}} u^{-1} \vph^{I'} .
  \end{eqnarray*}
\end{Th}
By the results in Sections \ref{section-cycle} and \ref{section-cohomology}, 
twisted period relations (\ref{TPR-eq}) can be 
reduced to quadratic relations among $F_A$'s. 
We write out two of them as a corollary. 
\begin{Cor}\label{TPR-cor}
  We use the notations  
  \begin{align*}
    & b^{i_1 \cdots i_r} =b+\sum (1-c_{i_p})e_{i_p} ,\ 
    c^{i_1 \cdots i_r} =c+2\sum (1-c_{i_p})e_{i_p}  
    \quad ({\rm see \ Proposition \ \ref{solution}}), \\
    & a_{i_1 \cdots i_r} :=a+r-\sum c_{i_p} ,\\ 
    &\tilde{b}^{i_1 \cdots i_r} :=(1,\ldots ,1) -b^{i_1 \cdots i_r},\ 
    \tilde{c}^{i_1 \cdots i_r} :=(2,\ldots ,2) -c^{i_1 \cdots i_r}.
  \end{align*}
  \begin{enumerate}[(i)]
  \item The equality (\ref{TPR-eq}) for $I=I'=\emptyset$ is reduced to 
    \begin{align*}
      & \frac{\prod (c_k-1)}{a} 
      \cdot \sum_{I} \left( A_I \prod_{j \not\in I} \frac{1}{b_j} \right) \\
      & = \sum _{I} \Bigg[ 
      \prod_q \frac{c_{j_q}-b_{j_q}-1}{b_{j_q}}
      \cdot \frac{1}{a_{i_1 \cdots i_r}} 
      \cdot F_A (a_{i_1 \cdots i_r} ,b^{i_1 \cdots i_r} ,c^{i_1 \cdots i_r} ;x) \\
      & \hspace{50mm} 
      \cdot F_A (-a_{i_1 \cdots i_r} ,-b^{i_1 \cdots i_r} ,\tilde{c}^{i_1 \cdots i_r} ;x) 
      \Biggr] .
    \end{align*}
  \item The equality (\ref{TPR-eq}) for $I=\emptyset ,\ I'=\{ 1,\ldots ,m \}$ is reduced to   
    \begin{align*}
      &  \frac{\prod (1-c_k )}{a} \cdot A_{1\cdots m}  \\
      & =\sum _{I} 
      \frac{(-1)^r}{a_{i_1 \cdots i_r}} 
      \cdot F_A (a_{i_1 \cdots i_r} ,b^{i_1 \cdots i_r} ,c^{i_1 \cdots i_r} ;x)  
      \cdot F_A (-a_{i_1 \cdots i_r},
      \tilde{b}^{i_1 \cdots i_r} ,\tilde{c}^{i_1 \cdots i_r} ;x) .
    \end{align*}
  \end{enumerate}
\end{Cor}
\begin{proof}
  We prove (i). 
  By Proposition \ref{series-cycle} and Theorem \ref{solutions-cycle}, we have
  \begin{align*}
    g_{i_1 \cdots i_r} &= e^{\pi \sqrt{-1} (\sum b_{i_p} -\sum c_{i_p} +r)}
    \cdot \prod _{p=1} ^{r} \Ga (c_{i_p} -1) 
    \cdot \prod _{q=1} ^{m-r} \frac{\Ga(b_{j_q}) \Ga(c_{j_q}-b_{j_q})}{\Ga(c_{j_q})} \\ 
    & \cdot \frac{\Ga (1-a)}{\Ga (\sum c_{i_p}-a-r+1)} 
    \cdot \prod _{p=1} ^{r} x_{i_p} ^{1-c_{i_p}} \cdot 
    F_A (a+r-\sum c_{i_p} ,b^{i_1 \cdots i_r} ,c^{i_1 \cdots i_r};x).
  \end{align*}
  On the other hand, we can express $g_{i_{1} \cdots i_{r}} ^{\vee}$ like this  
  by the replacement 
  $$
  (a,b,c) \longmapsto (-a,-b,(2,\ldots ,2)-c),
  $$ 
  since $u^{-1}\vph$ is written as 
  $$
  u^{-1}\vph = \prod t_k ^{-b_k-1} (1-t_k)^{-c_k+b_k+1} \cdot (1-\sum x_k t_k )^a 
  dt_1 \wedge \cdots \wedge dt_m.
  $$
  Thus we obtain 
  \begin{align*}
    g_{i_{1} \cdots i_{r}}^{\vee} &=e^{\pi \sqrt{-1} (-\sum b_{i_p}+\sum c_{i_p}-r)} \\
    &\cdot \prod _{p=1} ^{r} \Ga (1-c_{i_p}) 
    \cdot \prod _{q=1} ^{m-r} \frac{\Ga(-b_{j_q}) \Ga(2-c_{j_q}+b_{j_q})}{\Ga(2-c_{j_q})} 
    \cdot \frac{\Ga (1+a)}{\Ga (-\sum c_{i_p}+a+r+1)} \\
    &\cdot \prod _{p=1} ^{r} x_{i_p} ^{c_{i_p}-1} \cdot 
    F_A (-a-r+\sum c_{i_p} ,-b^{i_1 \cdots i_r} ,(2,\ldots ,2)-c^{i_1 \cdots i_r};x).
  \end{align*}
  By the formula (\ref{Gamma}) and Theorem \ref{H-intersection}, we have 
  \begin{align*}
    & \frac{\Ga(1-a)\Ga(1+a)}{\Ga(\sum c_{i_p}-a-r+1)\Ga(-\sum c_{i_p}+a+r+1)}
    \cdot \prod_p \Ga(c_{i_p}-1)\Ga(1-c_{i_p}) \\
    & \cdot \prod_q \frac{\Ga(b_{j_q})\Ga(-b_{j_q}) \cdot \Ga(c_{j_q}-b_{j_q})\Ga(1-c_{j_q}+b_{j_q})}
    {\Ga(c_{j_q})\Ga(2-c_{j_q})}\\
    &=(2\pi \sqrt{-1})^m \cdot 
    \prod_k \frac{1}{c_k-1} 
    \cdot \prod_q \frac{c_{j_q}-b_{j_q}-1}{b_{j_q}}
    \cdot \frac{a}{a+r-\sum c_{i_p}}
    \cdot I_{h} (\De _{i_{1} \cdots i_{r}} ,\De _{i_{1} \cdots i_{r}}^{\vee} ) .
  \end{align*}
  Hence, we obtain (i) by Proposition \ref{C-intersection}. 
  A similar calculation shows (ii).  
\end{proof}

\section*{Acknowledgments}
The author thanks Professor Keiji Matsumoto for his useful advice and constant encouragement.

\end{document}